\input xy
\xyoption{all}\xyoption{poly}
\newdir{ >}{{}*!/-9pt/@{>}}

\documentclass[reqno]{amsart}
\usepackage[T1]{fontenc}
\usepackage[utf8]{inputenc}

\usepackage{lineno,hyperref}
\modulolinenumbers[5]

\usepackage{tikz}
\usetikzlibrary{decorations.pathreplacing,matrix,positioning}
\usepackage{amsmath,latexsym,amssymb,mathrsfs,amsthm} 

\usepackage{textcomp}
\usepackage{graphicx}

\usepackage{hyperref}
\usepackage[all]{xy}
\newcommand{\catC}{\mathcal{C}}

\newcommand{\catX}{\mathcal{X}}
\newcommand{\catV}{\mathcal{V}}
\newcommand{\Qnd}{\mathsf{Qnd}}

\newcommand{\Qnds}{\mathsf{SymQnd}}
\newcommand{\Qndsa}{\mathsf{AbSymQnd}}

\newcommand{\Eq}{\mathsf{Eq}}
\newcommand{\lhdi}{\lhd^{-1}}
\newtheorem{theorem}{Theorem}[section]
\newtheorem{lemma}[theorem]{Lemma}
\newtheorem{proposition}[theorem]{Proposition}
\newtheorem{corollary}[theorem]{Corollary}
\theoremstyle{definition}
\newtheorem{definition}[theorem]{Definition}
\theoremstyle{remark}
\newtheorem{remark}[theorem]{Remark}

\usepackage{todonotes}
\usepackage{graphicx}
 \usepackage{color}

\newdir{ >}{{}*!/-10pt/\dir{>}}
\def\cartesien{%
    \ar@{-}[]+R+<6pt,-1pt>;[]+RD+<6pt,-6pt>%
    \ar@{-}[]+D+<1pt,-6pt>;[]+RD+<6pt,-6pt>%
  }

\begin{document}
  \title{A characterization of central extensions in the variety of quandles}
	  \author{Val\'erian Even}
	  \address{Institut de Recherche en Math\'ematique et Physique, \\  Universit\'e catholique de Louvain, \\ Chemin du Cyclotron 2, 1348 Louvain-la-Neuve, Belgium.}
	  \author{Marino Gran}
	  \author{Andrea Montoli}
	  \address{Departamento de Matem\'atica \\
	  	Universidade de Coimbra \\
	  	Apartado 3008, 3001-454 Coimbra, Portugal}

\maketitle

\begin{abstract}
The category of symmetric quandles is a Mal'tsev variety whose
subvariety of abelian symmetric quandles is the category of
abelian algebras. We give an algebraic description of the quandle
extensions that are central for the adjunction between the variety
of quandles and its subvariety of abelian symmetric quandles.
\end{abstract}

\section{Introduction}

A \emph{quandle} \cite{Joyce} is a set $A$ equipped with two
binary operations $\lhd$ and $\lhdi$ such that the following
identities hold (for all $a,\ b,\ c \in A$):
\begin{itemize}
\item[(A1)] $a \lhd a = a = a \lhdi a$ (idempotency); \item[(A2)]
$(a \lhd b) \lhdi b = a = (a \lhdi b) \lhd b$ (right
invertibility); \item[(A3)] $(a \lhd b) \lhd c = (a \lhd c) \lhd
(b \lhd c)$ and $(a \lhdi b) \lhdi c = (a \lhdi c) \lhdi (b \lhdi
c)$ (self-distributivity).
\end{itemize}

This structure is of interest in knot theory, since the three
axioms above correspond to the Reidemeister moves on oriented link
diagrams. From a purely algebraic viewpoint, quandles capture the
properties of group conjugation: given a group $(G, \cdot, 1)$, by
defining the operations $a \lhd b = b \cdot a \cdot b^{-1}$ and $a
\lhdi b = b^{-1} \cdot a \cdot b$ on the underlying  set $G$ one
gets a quandle structure.

Quandles and quandle homomorphisms form a category denoted $\Qnd$.
This category, being a variety in the sense of universal
algebra~\cite{BS}, is an exact category (in the sense of
Barr~\cite{Barr}). The variety $\Qnd$ has some interesting
categorical properties, as recently observed in \cite{EG, EG2,
Bourn}. The present work continues this line of research, by
investigating the properties of the adjunction between the variety
of quandles and its subvariety $\Qndsa$ of abelian symmetric
quandles, in particular from the viewpoint of the categorical
theory of central extensions \cite{JK}.

The variety $\Qndsa$ of abelian symmetric quandles is the
subvariety of $\Qnd$ determined by the two additional identities
$$a \lhd b = b \lhd a$$ and $$(a \lhd b) \lhd (c \lhd d) = (a \lhd
c) \lhd (b \lhd d).$$ $\Qndsa$ is a Mal'tsev variety (actually
even a naturally Mal'tsev one \cite{John}, see Section
\ref{abeliansymm}), and it turns out to be an \emph{admissible}
subvariety of $\Qnd$: this fact guarantees the validity of a
Galois theorem of classification of the corresponding central
extensions (see \cite{J, JK}).

This is particularly interesting by keeping in mind that the
variety $\Qnd$ is not congruence modular, since it contains the
variety of sets as a subvariety. However, the subvariety $\Qndsa$
of abelian symmetric quandles yields an adjunction
\begin{equation}\label{theadjunction}\vcenter{\hbox{\begin{tikzpicture}[scale=0.3]
            \node[] (X) at (-5,0) {$\Qnd$};
            \node[] (Y) at (5,0) {$\Qndsa$};
            \node[line width=4pt] (C) at (0,0) {$\perp$};
            \draw[->] (X) to [bend left=25] node[above]{$I$} (Y);
            \draw[<-] (X) to [bend right=25] node[below]{$H$} (Y);
            \end{tikzpicture}}}
            \end{equation}
that is similar to the classical one
\begin{equation}\label{modular}\vcenter{\hbox{\begin{tikzpicture}[scale=0.3]
            \node[] (X) at (-5,0) {$\catV$};
            \node[] (Y) at (5,0) {$\catV_{ab}$};
            \node[line width=4pt] (C) at (0,0) {$\perp$};
            \draw[->] (X) to [bend left=25] node[above]{$I$} (Y);
            \draw[<-] (X) to [bend right=25] node[below]{$U$} (Y);
            \end{tikzpicture}}}
\end{equation}
where ${\catV}$ is any congruence modular variety and
${\catV_{ab}}$ its subvariety of abelian algebras in the sense of
commutator theory \cite{FMK}. Many interesting results in the
categorical theory of central extensions discovered in the last
years actually concern subvarieties of Mal'tsev varieties (see
\cite{Ever}, for instance, and the references therein). The
example investigated in the present paper is then of a rather
different nature, and will be useful to establish some new
connections between algebraic quandle theory and categorical
algebra.

To explain the main result of this paper more precisely, let us
briefly recall how the categorical notions of trivial extension
and of central extension are defined in any variety ${\catV}$ with
respect to a chosen subvariety $\catX$ of ${\catV}$. A surjective
homomorphism $f \colon A \rightarrow B$ in ${\catV}$ is a
\emph{trivial extension} if the commutative square induced by the
units of the reflection
\[\begin{tikzpicture}[scale=0.45]
\node[] (M) at (-2,2) {$A$}; \node[] (X) at (2,2) {$HI(A)$};
\node[] (M') at (-2,-2) {$B$}; \node[] (X') at (2,-2) {$HI(B)$};
\draw[->] (M) to node[above]{$\eta_A$} (X); \draw[->] (M) to
node[left]{$f$} (M'); \draw[->] (X) to node[right]{$HI(f)$} (X');
\draw[->] (M') to node[below]{$\eta_B$} (X'); \draw (-0.4,1.8) --
(-0.4,0.8) -- (-1.8,0.8);
\end{tikzpicture}\]
is a pullback. A surjective homomorphism $f \colon A \rightarrow
B$ is a \emph{central extension} when there exists a surjective
homomorphism~$p \colon E \rightarrow B$ such that the extension
$\pi_1 \colon E \times_B A \rightarrow E$ in the pullback
\[\begin{tikzpicture}[scale=0.45]
\node[] (M) at (-2,2) {$E \times_B A$}; \node[] (X) at (2,2)
{$A$}; \node[] (M') at (-2,-2) {$E$}; \node[] (X') at (2,-2)
{$B$}; \draw[->] (M) to node[above]{$\pi_2$} (X); \draw[->] (M) to
node[left]{$\pi_1$} (M'); \draw[->] (X) to node[right]{$f$} (X');
\draw[->] (M') to node[below]{$p$} (X'); \draw (-0.2,1.8) --
(-0.2,0.8) -- (-1.8,0.8);
\end{tikzpicture}\]
 of $f$ along $p$ is a trivial extension. In any congruence
modular variety ${\catV}$ the central extensions defined in this
way, relatively to the adjunction \eqref{modular}, are precisely
the surjective homomorphisms $f \colon A \rightarrow B$ whose
kernel congruence $\mathsf{Eq}(f) = \{ (a_1, a_2) \in A \times A
\, \mid \, f(a_1)= f(a_2) \}$ is central in the sense of
commutator theory: $[\mathsf{Eq}(f) , A \times A] = \Delta_A$,
where $\Delta_A$ is the smallest congruence on $A$ (see \cite{G,
JK2}). In the present paper we characterize the central extensions
corresponding to the adjunction \eqref{theadjunction} as those
surjective quandle homomorphisms $f \colon A~\rightarrow~B$ such
that (a condition equivalent to) $[\mathsf{Eq}(f) , A \times A] =
\Delta_A$ holds and, moreover, each fiber $f^{-1}(b)= \{a \in A \,
\mid f(a)=b \}$ is an abelian symmetric quandle, for any $b\in B$
(Theorem \eqref{mainresult}). This latter property implies that
$f\colon A \to B$ is $\Sigma$-special in the terms of~\cite{Bourn}
and thus this work fits in the partial Mal'tsev context,
thoroughly studied in the preprint \cite{Bourn2}, that became
available on the ArXiv a few months after the present paper.

\section{Symmetric quandles and abelian symmetric quandles}\label{abeliansymm}

A quandle $A$ is symmetric if it satisfies the additional
identity:
\begin{equation}
a \lhd b = b \lhd a,
\end{equation}
for all $a,\ b \in A$. We write $\Qnds$ for the corresponding
category of symmetric quandles, which is then a subvariety of the
variety $\Qnd$ of all quandles. Here below we observe that the
category $\Qnds$ is a Mal'tsev variety \cite{Smith}, which will be
shown to be an admissible subcategory of $\Qnd$ for the
categorical theory of central extensions \cite{JK}.

\begin{proposition} \cite{Bourn}
The category $\Qnds$ is a Mal'tsev variety.
\end{proposition}

\begin{proof}
Let $p$ be the ternary term defined by \[{p(a,b,c)} = (a \lhd c)
\lhd^{-1} b.\] We then have the identities
\begin{align*}
p(a,a,b) &= (a \lhd b) \lhdi a = (b \lhd a) \lhdi a = b, \\
p(a,b,b) &= (a \lhd b) \lhdi b = a.
\end{align*}
\end{proof}

Recall that a quandle $A$ is \emph{abelian} \cite{Joyce} if it
satisfies the additional axiom
\[(a \lhd b) \lhd (c \lhd d) = (a \lhd c) \lhd (b \lhd d)\] for all $a,\ b,\ c,\ d \in A$.
Note that this axiom is equivalent to the following one:
\begin{equation}\label{abelian2}
(a \lhd b) \lhdi (c \lhd d) = (a \lhdi c) \lhd (b \lhdi d).
\end{equation}
\begin{remark}
Not all abelian quandles are symmetric. Indeed, recall that a
quandle $A$ is \emph{trivial} if $a\lhd b= a= a \lhdi b$ for all
$a,b \in A$. Any trivial quandle is abelian, but it is not
symmetric (as long as it has at least two elements).

Also, not all symmetric quandles are abelian. The smallest
symmetric quandle which is not abelian is a quandle of order $81$
and is constructed in~\cite{Soublin}.
\end{remark}

Let us write $\Qndsa$ for the category of abelian symmetric
quandles, \linebreak $U \colon \Qndsa \rightarrow \Qnds$ and $V
\colon \Qnds \rightarrow \Qnd$ for the inclusion functors. Since
$\Qndsa$ is a subvariety of $\Qnds$ and $\Qnds$ is a subvariety of
$\Qnd$, both these functors have left adjoints, denoted by $ab
\colon \Qnds \rightarrow \Qndsa$ and ${ sym} \colon \Qnd
\rightarrow \Qnds$, respectively:
\[\vcenter{\hbox{\begin{tikzpicture}[scale=0.6]
            \node[] (X) at (-5,0) {$\Qnd$};
            \node[] (Y) at (0,0) {$\Qnds$};
            \node[] (Z) at (5,0) {$\Qndsa$};
            \node[line width=4pt] (C) at (-2.5,0) {$\perp$};
            \node[line width=4pt] (C) at (2.5,0) {$\perp$};
            \draw[->] (X) to [bend left=25] node[above]{$sym$} (Y);
            \draw[<-] (X) to [bend right=25] node[below]{$V$} (Y);
            \draw[->] (Y) to [bend left=25] node[above]{$ab$} (Z);
            \draw[<-] (Y) to [bend right=25] node[below]{$U$} (Z);
            \end{tikzpicture}}}\]

We are now going to show that abelian symmetric quandles are the
internal Mal'tsev algebras in $\Qnds$.

\begin{definition}
An internal Mal'tsev algebra in a variety ${\catV}$ is an algebra
$A \in {\catV}$ with a homomorphism $p_A \colon A\times A \times A
\to A$ such that $p_A(a,a,b) = b$ and $p_A(a,b,b) = a$.
\end{definition}

Let us write $\mathsf{Mal}({\catV})$ for the category of internal
Mal'tsev algebras in ${\catV}$. In a Mal'tsev category, thus in
particular in the category $\Qnds$, any morphism preserves the
Mal'tsev operation (see Corollary $4.1$ in \cite{Gran}, for
instance): this means that the subcategory $\mathsf{Mal}(\Qnds)$
is full in $\Qnds$. The following observation has been found
independently by {Bourn} \cite{Bourn}:
\begin{theorem}
\[\Qndsa = \mathsf{Mal}(\Qnds).\]
\end{theorem}

\begin{proof}
Let $A \in \Qndsa$, and let $p_A \colon A \times A \times A \to A$
be the Mal'tsev operation on $A$ defined by ${ p_A(a,b,c)} = (a
\lhd c) \lhdi b$. We have to check that it is a quandle
homomorphism. For any $a, b, c, x, y, z \in A$ we have
\begin{align*}
p_A((a,b,c) \lhd (x,y,z)) &= {p_A(a \lhd x, b \lhd y, c \lhd z)} \\
&= \left( (a \lhd x) \lhd (c \lhd z) \right) \lhdi (b \lhd y) \\
&= \left( (a \lhd c) \lhd (x \lhd z) \right) \lhdi (b \lhd y) \\
&= \left( (a \lhd c) \lhdi b \right) \lhd ((x \lhd z) \lhdi y) \\
&= p_A(a,b,c) \lhd p_A(x,y,z).
\end{align*}
This shows that $A$ belongs to $\mathsf{Mal}(\Qnds)$.

Conversely, when $A \in \mathsf{Mal}(\Qnds)$, the unique internal
Mal'tsev operation on $A$ \cite{John} is necessarily given by (any
of) the Mal'tsev operations of the theory of the variety $\Qnds$.
Accordingly, it is defined by $p_A(a,b,c) = (a \lhd c) \lhdi b$,
and it is such that $p_A (a,b,a) = a \lhdi b$. Moreover, $p_A
\colon {A \times A \times A} \rightarrow A$ preserves the binary
operation $\lhd$, so that the equality $$ p_A((a,b,a) \lhd
(x,y,x)) = p_A(a,b,a) \lhd p_A(x, y,x) $$ gives
$$(a \lhd x) \lhdi (b \lhd y) =  (a \lhdi b) \lhd (x \lhdi y).$$
This is precisely the identity \eqref{abelian2}, and the quandle
$A$ belongs to $\Qndsa$.
\end{proof}

We now recall the definition of two classes of morphisms in
$\Qnd$, first investigated by Bourn, that will be important for
our work:

\begin{definition} \cite{Bourn}
We denote by $\Sigma$ the class of split epimorphisms $f \colon A
\rightarrow B$ with a given section $s \colon B \rightarrow A$
(i.e. $f \circ  s = 1_B$) in the category $\Qnd$ such that the map
\linebreak $s(b) \lhd - \colon f^{-1}(b) \to f^{-1}(b)$ is
surjective, for any $b\in B$.
\end{definition}
In other words, the split epimorphism $f$ with section $s$ is in
$\Sigma$ if, for any $b \in B$ and $a \in  f^{-1}(b)$, there is a
$k_a \in f^{-1}(b)$ such that $s(b) \lhd k_a = a$.
\begin{remark}
This element $k_a$ also depends on $b$, so that one should write
$k_{b, a}$, instead. We shall simply write $k_a$, however, to
simplify the notations.
\end{remark}
Given an internal equivalence relation $(R,r_1, r_2)$ in $\Qnd$ on
$A$, i.e. a congruence on $A$, we write $\delta_R \colon {A}
\rightarrow R$ for the homomorphism defined by $\delta_R (a)=
(a,a)$, for any $a$ in $A$. An equivalence relation $(R,r_1, r_2)$
is said to be a \emph{$\Sigma$-equivalence relation} if the split
epimorphism $r_1 \colon R \rightarrow {A}$ with section $\delta_R
\colon {A} \rightarrow R$ belongs to the class $\Sigma$.

Given a quandle homomorphism $f \colon A \rightarrow B$, we write
$(\Eq (f), f_1, f_2)$ for the kernel pair of $f$, where $f_1
\colon \Eq (f) \rightarrow A$ and $f_2  \colon \Eq (f) \rightarrow
A$ are the canonical projections: in a variety of universal
algebras $\Eq (f)$ is simply the kernel congruence on $A$ defined
by $\Eq(f) = \{ (a_1, a_2) \in A \times A \mid f(a_1)= f(a_2) \}.$
\begin{definition} {\cite{BMMS, BMMS2, Bourn}}
A morphism $f \colon A \rightarrow B$ in $\Qnd$ is
\emph{$\Sigma$-special} if $(\Eq(f), f_1, f_2)$ is a
$\Sigma$-equivalence relation.
\end{definition}
The following result is a direct consequence of Theorem $3.9$ in
\cite{Bourn}, and will be useful later on:
\begin{theorem}
Let $f \colon A \rightarrow B$ be a $\Sigma$-special homomorphism
in $\Qnd$. Then any congruence $R$ on $A$ permutes with $\Eq(f)$
in the sense of the composition of relations:
$$ R \circ \Eq(f) = \Eq(f) \circ R.$$
\end{theorem}
\begin{corollary}\label{surjective}
Given a pushout of surjective homomorphisms
\[\begin{tikzpicture}[scale=0.35]
\node[] (M) at (-2,2) {$A$}; \node[] (X) at (2,2) {$B$}; \node[]
(M') at (-2,-2) {$C$}; \node[] (X') at (2,-2) {$D$}; \draw[->] (M)
to node[above]{$f$} (X); \draw[->] (M) to node[left]{$g$} (M');
\draw[->] (X) to node[right]{$h$} (X'); \draw[->] (M') to
node[below]{$l$} (X');
\end{tikzpicture}\]
where $f$ is $\Sigma$-special, {the} induced homomorphism $A
\xrightarrow{(g,f)} C \times_D B$ to the pullback is surjective.
\end{corollary}
\begin{proof}
The proof is essentially the same as the one given in \cite{EG},
Lemma $1.7$ (which is adapted from \cite{CKP}).
\end{proof}

\section{Central extensions in the category of quandles}
If $\catC$ is a finitely complete category, a double equivalence
relation $C$ in $\catC$ is an equivalence relation internal in the
category of equivalence relations in $\catC$. It can be
represented by a diagram
\begin{equation}\label{double}\vcenter{\hbox{\begin{tikzpicture}[scale=0.45]
    \node[] (C) at (-2,2) {$C$};
    \node[](R) at (-2,-2) {$R$};
    \node[](S) at (2,2) {$S$};
    \node[](X) at (2,-2) {$A$};
    \draw[->,transform canvas ={xshift=-0.5ex}] (C) to node[left]{$\pi_1$} (R);
    \draw[->,transform canvas ={xshift=0.5ex}] (C) to node[right]{$\pi_2$} (R);
    \draw[->,transform canvas ={xshift=-0.5ex}] (S) to node[left]{$s_1$} (X);
    \draw[->,transform canvas ={xshift=0.5ex}] (S) to node[right]{$s_2$} (X);
    \draw[->,transform canvas ={yshift=-0.5ex}] (C) to node[below]{$p_2$} (S);
    \draw[->,transform canvas ={yshift=0.5ex}] (C) to node[above]{$p_1$} (S);
    \draw[->,transform canvas ={yshift=-0.5ex}] (R) to node[below]{$r_2$} (X);
    \draw[->,transform canvas ={yshift=0.5ex}] (R) to node[above]{$r_1$} (X);
    \end{tikzpicture}}}
    \end{equation}
where $r_1 \circ \pi_1 = s_1 \circ p_1$, $r_1 \circ \pi_2 = s_2
\circ p_1$, $r_2 \circ \pi_1 = s_1 \circ p_2$ and $r_2 \circ \pi_2
= s_2 \circ p_2$. In this case one usually says that $C$ is a
double equivalence relation on the equivalence relations $R$ and
$S$.
\begin{definition}
Given equivalence relations $R$ and $S$ on $A$, a double
equivalence relation $C$ on $R$ and $S$ (as in \eqref{double}) is
called a \emph{centralizing relation} when the square
\[\begin{tikzpicture}[scale=0.35]
    \node[] (C) at (-2,2) {$C$};
    \node[](R) at (-2,-2) {$R$};
    \node[](S) at (2,2) {$S$};
    \node[](X) at (2,-2) {$A$};
    \draw[->] (C) to node[above]{$p_2$} (S);
    \draw[->] (C) to node[left]{$\pi_1$} (R);
    \draw[->] (R) to node[below]{$r_2$} (X);
    \draw[->] (S) to node[right]{$s_1$} (X);
    \end{tikzpicture}\]
is a pullback.
\end{definition}

\begin{definition}\label{defconnector}
A \emph{connector} between $R$ and $S$ is an arrow $p \colon R
\times_A S \to A$ such that \[\begin{tabular}{ll}
1. $p(x,x,y) = y$ & 1'. $p(x,y,y) = x$ \\
2. $x Sp(x,y,z)$ & 2'. $zRp(x,y,z)$ \\
3. $p(x,y,p(y,u,v)) = p(x,u,v)$ & 3'. $p(p(x,y,u),u,v) = p(x,u,v)$
\end{tabular}\]
\end{definition}

In the Mal'tsev context \cite{BG} the existence of a connector
between $R$ and $S$ is already guaranteed by the existence of a
\emph{partial Mal'tsev operation} $p \colon R \times_A S \to A$,
i.e. when the identities $p(x,x,y) = y$ and $p(x,y,y) = x$ in
Definition \ref{defconnector} are satisfied. Accordingly, in a
Mal'tsev category the existence of a double centralizing relation
on $R$ and $S$ is equivalent to the existence of a partial
Mal'tsev operation. Moreover, a connector is unique, when it
exists: accordingly, for two given equivalence relations, having a
connector becomes a property.

In a Mal'tsev variety a congruence $R$ on an algebra $A$ is called
\emph{algebraically central} if there is a centralizing double
relation on $R$ and $A \times A$, this latter being the largest
equivalence relation on $A$. In terms of commutators, this fact is
expressed by the condition $[R, A \times A]= \Delta_A$.

Given a surjective homomorphism $f \colon A \to B$ in the variety
$\Qnd$ of quandles such a connector between the kernel pair
$\Eq(f)$ and $A \times A$ may not be unique, or may not exist at
all. However, there is a special class of homomorphisms for which
such a connector is unique when it exists.


Given a homomorphism $f\colon A \rightarrow B$ in $\Qnd$, each
fiber $f^{-1}(b)$ (for $b \in B$) is a subquandle of $A$.
 We shall say that \emph{$f$ has abelian symmetric fibers} if $f^{-1}(b) \in \Qndsa$, for all $b \in B$.

 \begin{proposition} \label{symmetric implies special}
    If $f \colon A \to B$ has symmetric fibers, then it is
    $\Sigma$-special.
 \end{proposition}
 \begin{proof}
    Consider the kernel pair of $f$
    \[\vcenter{\hbox{\begin{tikzpicture}[scale=0.45]
            \node[] (C) at (-2,2) {$\Eq(f)$};
            \node[](A) at (-2,-2) {$A$};
            \node[](a) at (2,2) {$A$};
            \node[](B) at (2,-2) {$B$};
            \draw[->,transform canvas ={xshift=-0.5ex}] (C) to node[left]{$f_1$} (A);
            \draw[<-,transform canvas ={xshift=0.5ex}] (C) to node[right]{$\delta_f$} (A);
            \draw[->] (A) to node[below]{$f$} (B);
            \draw[<-,transform canvas ={yshift=-0.5ex}] (C) to node[below]{$\delta_f$} (a);
            \draw[->,transform canvas ={yshift=0.5ex}] (C) to node[above]{$f_2$} (a);
            \draw[->] (a) to node[right]{$f$} (B);
            \end{tikzpicture}}}
    \]
    One has to check
    that $(f_1,\delta_f)$ is in $\Sigma$. Let $a \in A$ and $(a,a')
    \in f_1^{-1}(a)$, then in particular $f(a) = f(a')$, so that $a'
    \lhdi a$ is such that $f(a) = f(a' \lhdi a)$. It follows that $(a,
    a' \lhdi a) \in \Eq(f)$, and then
    \[(a,a) \lhd (a, a' \lhdi a) =
    (a \lhd a , a \lhd (a' \lhdi a)) = (a, (a' \lhdi a) \lhd a ) =
    (a,a').\]
 \end{proof}

 \begin{remark} \label{uniqueness}
    When a split epimorphism $f \colon A \to B$ with
    section $s \colon {B \to A}$ has symmetric fibers, then
    $s(b) \lhd - \colon f^{-1} (b) \rightarrow f^{-1} (b)$ is always
    injective: if $x \in f^{-1} (b) $ and $y \in f^{-1} (b) $ are such
    that $s(b) \lhd x =s(b) \lhd y$, since $s(b) \in f^{-1} (b)$, we
    get $x \lhd s(b) = y \lhd s(b)$, and hence $x=y$ by right
    invertibility.
 \end{remark}

\begin{lemma}\label{pullback}
Consider the following pullback
\[\begin{tikzpicture}[scale=0.45]
\node[] (M) at (-2,2) {$E \times_B A$}; \node[] (X) at (2,2)
{$A$}; \node[] (M') at (-2,-2) {$E$}; \node[] (X') at (2,-2)
{$B.$}; \draw[->] (M) to node[above]{$\pi_2$} (X); \draw[->] (M)
to node[left]{$\pi_1$} (M'); \draw[->] (X) to node[right]{$f$}
(X'); \draw[->] (M') to node[below]{$p$} (X'); \draw (-0.2,1.8) --
(-0.2,0.8) -- (-1.8,0.8);
\end{tikzpicture}\]
If $f \colon A \to B$ has abelian symmetric fibers then so does
$\pi_1 \colon E \times_B A \to E$. Moreover, if $p \colon E \to B$
is a surjective homomorphism, then $f \colon A \to B$ has abelian
symmetric fibers if $\pi_1 \colon E \times_B A \to E$ has abelian
symmetric fibers.
\end{lemma}

\begin{proof}
The first assertion follows from the fact that if $(e,a) \in E
\times_B A$ then the fibers $\pi_1^{-1} (e)$ and $f^{-1} (f(a))$
are isomorphic. The proof of the second assertion is similar, the
surjectivity of $p$ guaranteeing that, for any $a \in A$, there
exists $e \in E$ such that $(e,a) \in E \times_B A$.
\end{proof}

\begin{lemma}\label{jointepi} \cite{Bourn}
Let $f \colon A \rightarrow B$ be a split epimorphism, with
section $s \colon B \rightarrow A$, in $\Sigma$. Consider the
following pullback of $f$ along a split epimorphism $p \colon E
\to B$, with section $t \colon B \rightarrow E$:
\[\vcenter{\hbox{\begin{tikzpicture}[scale=0.5]
    \node[] (C) at (-3,2) {$E \times_B A$};
    \node[](R) at (-3,-2) {$E$};
    \node[](S) at (3,2) {$A$};
    \node[](X) at (3,-2) {$B.$};
    \draw[->,transform canvas ={xshift=-0.5ex}] (C) to node[left]{$\pi_1$} (R);
    \draw[<-,transform canvas ={xshift=0.5ex}] (C) to node[right]{$(1_E, s \circ p)$} (R);
    \draw[->,transform canvas ={xshift=-0.5ex}] (S) to node[left]{$f$} (X);
    \draw[<-,transform canvas ={xshift=0.5ex}] (S) to node[right]{$s$} (X);
    \draw[->,transform canvas ={yshift=-0.5ex}] (C) to node[below]{$\pi_2$} (S);
    \draw[<-,transform canvas ={yshift=0.5ex}] (C) to node[above]{$(t \circ f,1_A)$} (S);
    \draw[->,transform canvas ={yshift=-0.5ex}] (R) to node[below]{$p$} (X);
    \draw[<-,transform canvas ={yshift=0.5ex}] (R) to node[above]{$t$} (X);
    \end{tikzpicture}}}
    \]
Then $(1_E,s \circ p)$ and $(t \circ f , 1_A)$ are jointly
epimorphic.
\end{lemma}

\begin{proof}
Let $(e,a) \in E \times_B A$; we shall show that $(e,a)$ can be
rewritten as a product of two elements in the images of $(1_E,s
\circ p)$ and $(t \circ f , 1_A)$, respectively. Since the split
 epimorphism $f $ is in $\Sigma$, there exists an element $k_a \in f^{-1}
(f(a))$ such that $sf(a) \lhd k_a = a$. Also, we always have $e =
(e \lhd^{-1} t p(e)) \lhd t p (e)$. Accordingly, by using the fact
that $f(a) = f(k_a)$ and $p(e)=f(a)$, we see that
\begin{align*} (e,a) &= ((e \lhdi tp(e)) \lhd tp(e), sf(a) \lhd k_a) \\
&= (e \lhdi tp(e),sf(a)) \lhd (tp(e),k_a) \\
&= (e \lhdi tp(e), sp(e)) \lhd (tf(k_a),k_a) \\
&= (e \lhdi tp(e), sp(e \lhdi tp(e))) \lhd (tf(k_a),k_a) \\
&= (1_E, s \circ p) (e \lhdi { t}p(e)) \lhd (t \circ f , 1_A)
(k_a).
\end{align*}
\end{proof}

\begin{corollary}
Let $R$ be an equivalence relation and $S$ be a
$\Sigma$-equivalence relation on the same quandle $A$ in $\Qnd$.
If there is a connector on $R$ and $S$, then it is unique.
\end{corollary}

\begin{proof}
This follows directly from Lemma~\ref{jointepi}.
\end{proof}

\begin{lemma}\label{connector}
Let $R$ be an equivalence relation and $S$ be a
$\Sigma$-equivalence relation on the same quandle $A$. For a
homomorphism $p \colon R \times_A S \to A,$ the following
conditions are equivalent :
\begin{enumerate}
\item $p$ is a partial Mal'tsev operation: $p(x,y,y) = x$ and
$p(x,x,y) = y$; \item $p$ is a connector between $R$ and $S$.
\end{enumerate}
\end{lemma}

\begin{proof}
This result is easily checked and also follows from
Lemma~\ref{jointepi}.
\end{proof}

From now on, we shall say that a surjective homomorphism with
abelian symmetric fibers $f \colon A \to B$ in $\Qnd$ is an
\emph{algebraically central extension} if its kernel congruence
$\Eq(f)$ is algebraically central: there is a connector between
$\Eq(f)$ and $A \times A$.

\begin{lemma}\label{decomposition}
Let $f \colon A \to B$ be an algebraically central extension with
abelian symmetric fibers, then $\Eq(f)$ is isomorphic to a product
$Q \times A$, where $Q$ is an abelian symmetric quandle.
\end{lemma}

\begin{proof}
Let $C$ be the centralizing relation on $\Eq(f)$ and $A \times A$;
consider the following diagram
\[\vcenter{\hbox{\begin{tikzpicture}[scale=0.4]
    \node[] (C) at (-2,2) {$C$};
    \node[](R) at (-2,-2) {$\Eq(f)$};
    \node[](S) at (2,2) {$A \times A$};
    \node[](X) at (2,-2) {$A$};
    \node[] (B) at (6,-2) {$B$};
    \node[] (Q) at (-2,-6) {$Q$};
    \node[] (U) at (2,-6) {$1$};
    \draw[->,transform canvas ={xshift=-0.5ex}] (C) to node[left]{$c_1$} (R);
    \draw[->,transform canvas ={xshift=0.5ex}] (C) to node[right]{$c_2$} (R);
    \draw[->,transform canvas ={xshift=-0.5ex}] (S) to  (X);
    \draw[->,transform canvas ={xshift=0.5ex}] (S) to  (X);
    \draw[->,transform canvas ={yshift=-0.5ex}] (C) to (S);
    \draw[->,transform canvas ={yshift=0.5ex}] (C) to (S);
    \draw[->,transform canvas ={yshift=-0.5ex}] (R) to (X);
    \draw[->,transform canvas ={yshift=0.5ex}] (R) to  (X);
    \draw[->] (X) to node[above]{$f$} (B);
    \draw[->] (R) to node[left]{$q$} (Q);
    \draw[->] (X) to (U);
    \draw[->,transform canvas={yshift=0.5ex}] (Q) to (U);
    \draw[->,transform canvas={yshift=-0.5ex}] (Q) to (U);
    \end{tikzpicture}}}
    \]
where $q$ is the coequalizer of $c_1$ and $c_2$. By the Barr-Kock
theorem \cite{Barr, BournGran}, the lower squares are pullbacks.
By Lemma~\ref{pullback}, the homomorphism $Q \to 1$ has abelian
symmetric fibers, hence $Q$ is an abelian symmetric quandle.
\end{proof}

The results in~\cite{Bourn} will be useful to show that the
category of abelian symmetric quandles is admissible with respect
to the class of surjective homomorphisms in the category of
quandles. In the following we shall characterize categorically
central and normal extensions in $\Qnd$ with respect to the
adjunction between the category of quandles and the category of
abelian symmetric quandles:

\[\vcenter{\hbox{\begin{tikzpicture}[scale=0.3]
            \node[] (X) at (-5,0) {$\Qnd$};
            \node[] (Y) at (5,0) {$\Qndsa$};
            \node[line width=4pt] (C) at (0,0) {$\perp$};
            \draw[->] (X) to [bend left=25] node[above]{$I$} (Y);
            \draw[<-] (X) to [bend right=25] node[below]{$H$} (Y);
            \end{tikzpicture}}}\]

Observe that each component of the unit of the adjunction is a
surjective homomorphism, since {$\Qndsa$} is a subvariety of
{$\Qnd$}, thus in particular it is stable in {$\Qnd$} under
subalgebras. The following theorem shows that the functor $I$
preserves a certain type of pullbacks. This is equivalent to the
admissibility condition of the subvariety $\Qndsa$ of $\Qnd$.

\begin{theorem}\label{Admissible}
In the previous adjunction, the reflector $I \colon \Qnd \to
\Qndsa$ preserves all pullbacks in $\Qnd$ of the form
\begin{equation}\label{admpullback}
\vcenter{\hbox{\begin{tikzpicture}[scale=0.40] \node[] (M) at
(-2,2) {$P$}; \node[] (X) at (2,2) {$H(X)$}; \node[] (M') at
(-2,-2) {$A$}; \node[] (X') at (2,-2) {$H(Y)$}; \draw[->] (M) to
node[above]{$p_2$} (X); \draw[->] (M) to node[left]{$p_1$} (M');
\draw[->] (X) to node[right]{$\phi$} (X'); \draw[->] (M') to
node[below]{$f$} (X'); \draw (-0.2,1.8) -- (-0.2,0.8) --
(-1.8,0.8);
\end{tikzpicture}}}
\end{equation}
where  $\phi \colon H(X) \rightarrow H(Y)$ is a surjective
homomorphism lying in the subcategory $\Qndsa$ and $f \colon A \to
H(Y)$ is a surjective homomorphism.
\end{theorem}
\begin{proof}
Consider the following commutative diagram where:
\begin{itemize}
\item the square on the back is the given pullback, where $\phi
\colon H(X) \rightarrow H(Y)$ is a surjective homomorphism in the
subcategory $\Qndsa$; \item the universal property of~the unit
$\eta_{P} \colon P \to HI(P)$ induces a unique arrow \linebreak
$HI(p_2) \colon HI(P) \to H(X)$ with $HI(p_2) \circ \eta_{P} =
p_2$; \item the universal property of~the unit $\eta_{A} \colon A
\to HI(A)$ induces a unique arrow \linebreak $HI(f) \colon HI(A)
\to H(Y)$ with $HI(f) \circ \eta_{A} = f$; \item $(P',\pi_1,
\pi_2)$ is the pullback of $HI(p_1)$ along $\eta_A$.
\end{itemize}
\[\vcenter{\hbox{\begin{tikzpicture}[scale=0.7]
\node[] (P) at (-4,2) {$P$}; \node[] (HX) at (4,2) {$H(X)$};
\node[] (A) at (-4,-2) {$A$}; \node[] (HY) at (4,-2) {$H(Y)$};
\node[] (P') at (-1,0) {$P'$}; \node[] (HIP) at (1,0) {$HI(P)$};
\node[] (HIA) at (1,-4) {$HI(A)$}; \draw[->] (P) to
node[above]{$p_2$} (HX); \draw[->] (P) to node[left]{$p_1$} (A);
\draw[->] (HX) to node[right]{$\phi$} (HY); \draw[->] (A) to
node[below right=0pt and 10pt]{$f$} (1,-2) node[fill=white]{} to
(HY); \draw[->, dotted] (P) to node[below left]{$\gamma$} (P');
\draw[->] (P) to node[above right]{$\eta_P$} (HIP); \draw[->] (P')
to node[below right]{$\pi_1$} (A); \draw[->] (P') to
node[below]{$\pi_2$} (HIP); \draw[->] (HIP) to node[above
left]{$HI(p_2)$} (HX); \draw[->] (HIP) to node[above right=10pt
and 0pt]{$HI(p_1)$} (HIA); \draw[->] (A) to node[below
left]{$\eta_A$} (HIA); \draw[->] (HIA) to node[below
right]{$HI(f)$} (HY);
\end{tikzpicture}}}\]
The quandle homomorphism $p_1$ is $\Sigma$-special by
Lemma~\ref{pullback} since $\phi$ has abelian symmetric fibers,
thus the homomorphism $\gamma$ is surjective by
Corollary~\ref{surjective}. The fact that~$\pi_1 \circ \gamma =
p_1$ and $HI(p_2) \circ \pi_2 \circ \gamma = p_2$ implies that
$\gamma$ is also injective. Indeed, this latter property follows
from the fact that the pullback projections $p_1$ and $p_2$ are
jointly monomorphic. Accordingly, the arrow $\gamma$ is bijective,
thus an isomorphism. Since~$\eta_A$ is a surjective homomorphism
it follows that the right face of the diagram is a pullback (see
Proposition~$2.7$ in~\cite{JK}, for instance), and the pullback
\ref{admpullback} is preserved by the functor $I$, as desired.
\end{proof}

\begin{corollary}\label{product}
The functor $I$ preserves products of the type $A \times Q$ where
$Q$ is an abelian symmetric quandle and $A$ is any quandle.
\end{corollary}

\begin{proof}
Remark that $A \times Q$ is the following pullback
\[\vcenter{\hbox{\begin{tikzpicture}[scale=0.4]
\node[] (M) at (-2,2) {$A \times Q$}; \node[] (X) at (2,2) {$Q$};
\node[] (M') at (-2,-2) {$A$}; \node[] (X') at (2,-2) {$1$};
\draw[->] (M) to node[above]{$p_2$} (X); \draw[->] (M) to
node[left]{$p_1$} (M'); \draw[->] (X) to  (X'); \draw[->] (M') to
(X');
\end{tikzpicture}}}
\] where $1$ is the terminal object in $\Qnd$, i.e.
the trivial quandle with one element.
\end{proof}

\begin{lemma}\label{algebraicpullback}
Consider the following pullback
\begin{equation}\label{pullbackcentral}
\vcenter{\hbox{\begin{tikzpicture}[scale=0.4] \node[] (M) at
(-2,2) {$E \times_B A$}; \node[] (X) at (2,2) {$A$}; \node[] (M')
at (-2,-2) {$E$}; \node[] (X') at (2,-2) {$B.$}; \draw[->] (M) to
node[above]{$\pi_2$} (X); \draw[->] (M) to node[left]{$\pi_1$}
(M'); \draw[->] (X) to node[right]{$f$} (X'); \draw[->] (M') to
node[below]{$p$} (X');
\end{tikzpicture}}}
\end{equation}
 If $f$ is an algebraically central extension
with abelian symmetric fibers, then $\pi_1$ is an algebraically
central extension with abelian symmetric fibers.

Moreover, if $p \colon E \to B$ is a surjective homomorphism, then
$f$ is an algebraically central extension with abelian symmetric
fibers if $\pi_1$ is an algebraically central extension with
abelian symmetric fibers.
\end{lemma}

\begin{proof}
First remark that we already know that the property of having
abelian symmetric fibers is preserved and reflected by pullbacks
along surjective homomorphisms by Lemma~\ref{pullback}.

Let $f \colon A \to B$ be an algebraically central extension with
abelian symmetric fibers. Write $p_f \colon A \times \Eq(f) \to A$
for the connector between $A \times A$ and $\Eq(f)$. Define the
quandle homomorphism $p_{\pi_1} \colon (E \times_B A ) \times
\Eq(\pi_1) \to E \times_B A$ as $p_{\pi_1} \left((e,a),
(e',b),(e',c)\right) = (e,p_f(a,b,c))$. We have
\[p_{\pi_1}((e,a),(e',b),(e',b)) = (e, p_f(a,b,b)) = (e,a)\] and
\[p_{\pi_1}((e,a),(e,a),(e,b)) = (e, p_f(a,a,b)) = (e,b).\] It is
then a connector by Lemma~\ref{connector}.

Now let $\pi_1 \colon E \times_B A \to E$ be an algebraically
central extension with abelian symmetric fibers. Write $p_{\pi_1}
\colon (E \times_B A) \times \Eq(\pi_1) \to E \times_B A$ for the
connector between $(E \times_B A) \times (E \times_B A)$ and
$\Eq(\pi_1)$. The surjectivity of $p \colon E \to B$ implies the
surjectivity of the homomorphism $\widehat{\pi_2} \colon (E
\times_B A) \times \Eq(\pi_1) \to A \times \Eq(f)$ defined by
\[\widehat{\pi_2} ((e,a),(e',b),(e',c)) = (a,b,c).\] First let us
show that $\Eq(\widehat{\pi_2}) \subset \Eq(\pi_2 \circ
p_{\pi_1})$. Let $$(((e_0,a),(e_0',b),(e_0',c)),
((e_1,a),(e_1',b),(e_1',c))) \in \Eq(\widehat{\pi_2}).$$ Since $f$
has abelian symmetric fibers by Lemma \ref{pullback}, it is
$\Sigma$-special by Proposition \ref{symmetric implies special}.
This means that the split epimorphism $\xymatrix{ \Eq(f)
\ar@<-2pt>[r]_-{f_1} & A \ar@<-2pt>[l]_-{{\delta_f} }}$ is in
$\Sigma$. In other terms, for all $b \in A$ and all $(b, c) \in
f_1^{-1}(b)$ there exists $k_{(b,c)} \in f_1^{-1}(b)$, where
$k_{(b,c)} = (b,k_c)$, such that $(b,b) \lhd k_{(b,c)} = (b,c)$.
Such a $k_{(b,c)}=  (b,k_c)$ is unique by Remark \ref{uniqueness}:
it follows that, for any $(b,c) \in \Eq(f)$, the element $k_c \in
A$ such that $f(k_c) = f(b) = f(c)$ and $b \lhd k_c = c$ is
unique. Then, for $i \in \{0,1\}$, we have
\[((e_i,a),(e_i',b),(e_i',c)) = ((e_i,a) \lhdi (e_i',b)
,(e_i',b),(e_i',b)) \lhd ((e_i',b),(e_i',b),(e_i',k_c)).\]
Consequently we remark that
\begin{align*}
&\pi_2 \circ p_{\pi_1} ((e_i,a),(e_i',b),(e_i',c)) \\
&=\pi_2 \circ p_{\pi_1} (((e_i,a) \lhdi (e_i',b) ,(e_i',b),(e_i',b)) \lhd ((e_i',b),(e_i',b),(e_i',k_c)))  \\
&=\pi_2 ( p_{\pi_1} ((e_i,a) \lhdi (e_i',b) ,(e_i',b),(e_i',b)) \lhd p_{\pi_1} ((e_i',b),(e_i',b),(e_i',k_c))  \\
&=\pi_2 (((e_i,a) \lhdi (e_i',b)) \lhd (e_i',k_c))  \\
&=\pi_2 ((e_i \lhdi e_i')\lhd e_i', (a \lhdi b) \lhd k_c) = (a
\lhdi b) \lhd k_c
\end{align*}
for both $i \in \{0,1\}$. This implies that $\Eq(\widehat{\pi_2})
\subset \Eq(\pi_2 \circ p_{\pi_1})$, and there is then a unique
quandle homomorphism $p_f \colon A \times \Eq(f) \to A$ such that
$p_f \circ \widehat{\pi_2} = \pi_2 \circ p_{\pi_1}$, i.e. $p_f
(a,b,c) = (a \lhdi b) \lhd k_c$ where $k_c$ is the unique element
such that $b \lhd k_c = c$ as above. Moreover, we have
\[p_f(a,b,b) = (a \lhd^{-1} b) \lhd b = a \] for $(a,b,b) \in A
\times \Eq(f)$ and \[p_f(a,a,b) = (a \lhd^{-1} a) \lhd k_b = a
\lhd k_b = b \] for $(a,a,b) \in A \times \Eq(f)$, so $p_f$ is a
connector by Lemma~\ref{connector}.
\end{proof}

Before stating our main result, we recall that a surjective
homomorphism $f \colon A \to B$ is a \emph{normal extension} when
the homomorphism $f_1$ in the pullback of $f$ along itself is a
trivial extension

\[\vcenter{\hbox{\begin{tikzpicture}[scale=0.4]
\node[] (M) at (-2,2) {$\Eq(f)$}; \node[] (X) at (2,2) {$A$};
\node[] (M') at (-2,-2) {$A$}; \node[] (X') at (2,-2) {$B.$};
\draw[->] (M) to node[above]{$f_2$} (X); \draw[->] (M) to
node[left]{$f_1$} (M'); \draw[->] (X) to node[right]{$f$} (X');
\draw[->] (M') to node[below]{$f$} (X'); \draw (-0.2,1.8) --
(-0.2,0.8) -- (-1.8,0.8)   ;
\end{tikzpicture}}}\]
 (see the Introduction for the definitions of trivial extension and of central extension).
\begin{theorem}\label{mainresult}
Given a surjective homomorphism $f \colon A \to B$ in $\Qnd$, the
following conditions are equivalent:
\begin{enumerate}
\item $f$ is an algebraically central extension with abelian
symmetric fibers; \item $f$ is a normal extension; \item $f$ is a
central extension.
\end{enumerate}
\end{theorem}

\begin{proof}
Let $f \colon A \to B$ be an algebraically central extension with
abelian symmetric fibers, then its kernel pair $\Eq(f)$ is
isomorphic to a product $Q \times A$ with $Q$ an abelian symmetric
quandle by Lemma~\ref{decomposition}. Corollary~\ref{product}
shows that $f$ is then a normal extension.

Every normal extension is a central extension.

Let $f \colon A \to B$ be a central extension. Then there is a
surjective homomorphism $p\colon E \rightarrow B$ such that the
first projection $\pi_1 \colon E {\times}_A B \rightarrow E$ in
the pullback \eqref{pullbackcentral} is a trivial extension. Then
$f \colon A \to B$ is an algebraically central extension with
abelian symmetric fibers by Lemma~\ref{algebraicpullback}, because
$\pi_1$ is the pullback of a morphism lying in $\Qndsa$.
\end{proof}

\begin{remark}
Note that there are surjective homomorphisms with abelian
symmetric fibers that are not algebraically central. Take for
instance the quandle $A$ given by the following table :

\[\begin{tabular}{c|cccc}

    $\lhd$&$a$&$b$&$c$&$d$\\
    \hline
    $a$&$a$&$c$&$b$&$a$\\

    $b$&$c$&$b$&$a$&$b$\\

    $c$&$b$&$a$&$c$&$c$\\

    $d$&$d$&$d$&$d$&$d$\\

\end{tabular}\] and the quandle homomorphism $f\colon A \to \{x,y\}$ defined by $f(a)=f(b)=f(c)=x$ and $f(d)=y$. Its kernel pair $\Eq(f)$ has $10$ elements, and thus can't be isomorphic to a product $A \times Q$ with $Q$ an abelian symmetric quandles since $A$ has $4$ elements.
\end{remark}

\begin{remark}
There are surjective algebraically central homomorphisms that do
not have symmetric fibers. Consider the additive group
$(\mathbb{Z}/2\mathbb{Z},+,\overline{0})$ and endow its underlying
set with the trivial quandle structure $a \lhd b = a$ for all $a,b
\in \mathbb{Z}/2\mathbb{Z}$. Remark that the group operation is a
quandle homomorphism:
\[(a \lhd b)+(a' \lhd b') = a+a' = (a+a') \lhd (b+b').\]
It follows that the Mal'tsev operation $p \colon
\mathbb{Z}/2\mathbb{Z} \times \mathbb{Z}/2\mathbb{Z} \times
\mathbb{Z}/2\mathbb{Z} \to \mathbb{Z}/2\mathbb{Z}$ defined by
$p(a,b,c) = a- b + c$ is a connector between the congruences
$\mathbb{Z}/2\mathbb{Z} \times \mathbb{Z}/2\mathbb{Z}$ and
$\mathbb{Z}/2\mathbb{Z} \times \mathbb{Z}/2\mathbb{Z}$. The
homomorphism $ \mathbb{Z}/2 \rightarrow 1$ is then an
algebraically central extension, whose (unique) fiber is not
symmetric, since $\overline{0} \lhd \overline{1} = \overline{0}
\neq \overline{1} = \overline{1} \lhd \overline{0}$.

\end{remark}

\section*{Acknowledgements}
This work was partially supported by the Centre for Mathematics of
the University of Coimbra -- UID/MAT/00324/2013, funded by the
Portuguese Government through FCT/MCTES and co-funded by the
European Regional Development Fund through the Partnership
Agreement PT2020, by the FCT grant number SFRH/BPD/69661/2010, and
by a FNRS grant \emph{Cr\'edit pour bref s\'ejour \`a
l'\'etranger} that has allowed the first author to visit the
Universidade de Coimbra for a research visit during which this
article was completed. The third author is a Postdoctoral
Researcher of the Fonds de la Recherche Scientifique-FNRS.

\end{document}